\documentclass[11pt,leqno]{amsart}
\usepackage[utf8]{inputenc}
\usepackage[utf8]{inputenc}
\usepackage{amsfonts,amssymb}
\usepackage{graphicx,tikz}
\usepackage{tikz-cd} 
\usepackage{float}
\usepackage{amsmath, amsthm, amsfonts}
\usepackage{geometry}
\usepackage{hyperref}
\usepackage{enumitem}  
\usepackage[capitalise]{cleveref}
\usepackage{comment}
\usepackage{enumitem}
\usepackage{amsrefs}
\usepackage{nicefrac}
\usepackage[official]{eurosym}
\usepackage{graphicx}
\usepackage{nomencl}
\usepackage{amsfonts}
\usepackage{hyperref}
\usepackage{amssymb}
\usepackage{fancyhdr}
\usepackage{amscd}
\usepackage[english]{babel}

\newtheorem{theorem}{Theorem}[section]
\newtheorem{thm}[theorem]{Theorem}
\newtheorem{lem}[theorem]{Lemma}
\newtheorem{pro}[theorem]{Proposition}
\newtheorem{cor}[theorem]{Corollary}

\newtheorem{rem}[theorem]{Remark}

\newtheorem*{thmA}{Theorem A}
\newtheorem*{thmB}{Theorem B}

\newtheorem{ques}{\scshape{Question}}

\newcommand{\Sol}{\mathrm{Sol}}
\newcommand{\PSL}{\mathrm{PSL}}
\newcommand{\Fit}{\mathrm{Fit}}
\newcommand{\A}{\mathrm{Alt}}

\title{On the solubilizer of an element in a finite group}
\author[B. Akbari]{B. Akbari}
\address{DMA, École normale supérieure, Université PSL, CNRS 75005 Paris, France
}
\email{(Akbari)  b.akbari@cornell.edu}
\author[C. Delizia]{C. Delizia}
\address{Dipartimento di Matematica, Universit\`a di Salerno, Fisciano (SA), Italy}
\email{(Delizia) cdelizia@unisa.it}
\author[C. Monetta]{C. Monetta}
\address{Dipartimento di Matematica, Universit\`a di Salerno, Fisciano (SA), Italy}
\email{(Monetta) cmonetta@unisa.it}


\begin{document}

\maketitle

\begin{abstract}
The solubility graph $\Gamma_S(G)$ associated with a finite group $G$ is a simple graph whose vertices are the elements of $G$, and there is an edge between two distinct vertices if and only if they generate a soluble subgroup. In this paper, we focus on the set of neighbors of a vertex $x$ which we call the solubilizer of $x$ in  $G$, $\Sol_G(x)$, investigating both arithmetic and structural properties of this set.
\end{abstract}


\section{Introduction}\label{I}
\noindent All groups considered in the present paper are supposed to be finite. The {\it solubility graph} $\Gamma_S(G)$ associated with a group $G$ is a simple graph whose vertices are the elements of $G$, and there is an edge between two distinct elements $x$ and $y$ if and only if the subgroup $\langle x,y \rangle$ is soluble. In \cite{thompson} Thompson proved that a finite group $G$ is soluble if and only if for every $x, y \in G$ the subgroup $\langle x,y \rangle$ is soluble. This implies that a finite group $G$ is soluble if and only if the graph $\Gamma_S (G)$ is complete.
We denote by $R(G)$ the soluble radical of a finite group  $G$, that is, the largest soluble normal subgroup of $G$. In \cite{GKPS} Guralnick et al. proved that if $x$ is an element of $G$, then $x \in R(G)$ if and only if the subgroup $\langle x,y \rangle$ is soluble for all $y \in G$. This means that $x \in R(G)$ if and only if $x$ is a {\it universal vertex} of $\Gamma_S(G)$, that is a vertex being adjacent to every other vertex in the graph. The subgraph of $\Gamma_S(G)$ obtained by removing all vertices in $R(G)$ is connected when $G$ is a finite group (see \cite{ALMM}). Furthermore, it has been proved in \cite{BLN} that its diameter is at most 5 in general, and at most 3 if $G$ is not almost simple.

For all $x \in G$, the neighborhood of $x$ in $\Gamma_S(G)$ is called the {\it solubilizer} of $x$ in $G$, and it is denoted by $\Sol_G(x)$. Then
\[
\Sol_G(x) = \{y \in G \ \mid \ \langle x,y \rangle \hbox{ is soluble}\}.
\]
The aim of this paper is to continue the investigation of properties of the solubilizer of an element in a finite group, started in \cite{HR} and \cite{ALMM}.

In Section~\ref{P} we collect known properties of the solubilizer of a vertex in the graph $\Gamma_S(G)$. In general $\Sol_G(x)$ is just a subset of $G$ and not a subgroup. However, it can happen that $\Sol_G(x)$ is a subgroup. For instance, the above result in \cite{GKPS} yields that  $x \in R(G)$ if and only if $\Sol_G(x)=G$. On the other side, if $\Sol_G(x)$ is a subgroup of $G$ it does not yield $x \in R(G)$.  Therefore the fact that a single solubilizer is a subgroup does not imply specific restrictions on the structure of the whole group. The situation is quite different when all solubilizers are subgroups. Indeed, in \cite{ALMM} it has been shown that a group $G$ is soluble if and only if $\Sol_G(x)$ is a subgroup of $G$ for all $x \in G$. 

It is an interesting problem to find algebraic conditions on the elements of a single solubilizer determining restrictions on the structure of the whole group. In \cite{ALMM} the authors proved that if $G$ is a group having an element $x$ such that the elements of $\Sol_G(x)$ commute pairwise, then $G$ is abelian. In Section~\ref{A} we show a generalization of this result. Given an integer $k \geq 2$ and elements $x_1, \ldots, x_k$ of a group $G$, the {\it long commutator of weight} $k$ is inductively defined by the formulae
\[
[x_1,x_2]=x_1^{-1}x_2^{-1}x_1x_2,  \qquad [x_1, \ldots, x_k]=[[x_1, \ldots, x_{k-1}], x_k]\quad (k>2).
\]
The subgroup $\gamma_k(G)$ generated by all long commutators of weight $k$ is the $k$th term of the lower central series of $G$. When rephrased in terms of long commutators, the above result in \cite{ALMM} asserts that $\gamma_2(G)=1$ if and only if there exists an element $x \in G$ such that $[x_1,x_2]=1$ for all $x_1,x_2 \in \Sol_G(x)$. Therefore the following question arises naturally.

\begin{ques}\label{q1}
Let $G$ be a group, $x \in G$ and $k \geq 3$. If $[u_1, \ldots, u_k]=1$ for every $u_1, \ldots, u_k \in \Sol_G(x)$, is $\gamma_k(G)=1$?
\end{ques}

Notice that a positive answer would provide a nilpotency criterion for finite groups, since the converse is obviously true. Our main result in Section~\ref{A} is an affirmative answer to Question~\ref{q1} when $k=3$. 

\begin{thmA}\label{thm:nilp2}
Let $G$ be a finite group. Then $G$ is nilpotent of class at most $2$ if and only if there exists an element $x \in G$ such that $[u_1,u_2,u_3]=1$ for every $u_1,u_2,u_3 \in \Sol_G(x)$.
\end{thmA}

Furthermore, for $k \geq 4$ we point out structural restrictions for a minimal (with respect to the cardinality) non-nilpotent group $G$ containing an element $x$ such that $[x_1, \ldots, x_k]=1$ for all $x_1, \ldots, x_k \in \Sol_G(x)$.

In \cite{ALMM} it has been shown that if $x$ is an element of an insoluble group $G$ then the cardinality of $\Sol_G(x)$ cannot be a prime. Furthermore, if $R(G)$ is not trivial, then the cardinality of $\Sol_G(x)$ cannot be a square of a prime. These are the first answers to the following arithmetic question related to the graph $\Gamma_S(G)$.

\begin{ques}\label{q2}
Let $G$ be a group, and let $x$ be an element of $G$. Which positive integers can occur as the cardinality of $\Sol_G(x)$?
\end{ques}

In Section~\ref{B} we prove that the above-mentioned result in \cite{ALMM} remains true if the hypothesis $R(G) \neq 1$ is dropped.

\begin{thmB}\label{thm:quadrato}
Let $G$ be an insoluble group and $x$ an element of $G$. Then the cardinality of $\Sol_G(x)$ cannot be equal to $p^2$ for any prime $p$.
\end{thmB}

Furthermore, we state restrictions for a prime $p$ when the cardinality of $\Sol_G(x)$ is equal to $3p$. Finally, we show that if $x$ is any element of an insoluble group $G$ then the set $\Sol_G(x)$ has cardinality at least $10$.


\section{Preliminary results}\label{P}

\noindent In this section we collect some properties and results which will be useful in the sequel. From now on, if $X$ is any subset of a finite group $G$, the cardinality of $X$ will be denoted by $|X|$. Moreover, for an element $x\in G$, $o(x)$ will denote the order of $x$, and $C_G(x)$ the centralizer of $x$ in $G$. Finally, for a subgroup $H$ of $G$, $N_G(H)$ and $C_G(H)$ will denote the normalizer and the centralizer of $H$ in $G$, respectively. We start by mentioning some results about the solubilizer of an element in a finite group.

\begin{lem}[\cite{ALMM}]\label{lem:property} Let $G$ be a group and $x \in G$. Then:
\begin{enumerate}
    \item[$(a)$] $ \langle x \rangle \subseteq C_G(x) \subseteq N_G(\langle x \rangle) \subseteq N_G(\langle x \rangle) \cup R(G) \subseteq \Sol_G(x)$;
    
    \item[$(b)$] $\Sol_G(x)$ is the union of all soluble subgroups of $G$ containing $x$;
    
    \item[$(c)$] $|\Sol_G(x)|$ is divisible by $o(x)$.
\end{enumerate}
\end{lem}

\begin{lem}[\cite{HR}]\label{lem:centralizer}
Let $G$ be a group and $x \in G$. Then $|C_G(x)|$ divides $|\Sol_G(x)|$.
\end{lem}

If $N$ is a normal soluble subgroup of $G$, we define
\[
\frac{\Sol_G(x)}{N} = \{yN \mid  y \in \Sol_G(x)\} = \{yN \mid \langle x,y \rangle \hbox{ is soluble} \}.
\]

In many situations, the following result enables to reduce to the case when the soluble radical is trivial.
\begin{lem}[\cite{HR}]\label{quotient}
If $N$ is a normal soluble subgroup of a group $G$, then $\Sol_G(x)$ is the union of cosets of $N$, and $\Sol_{G/N}(xN)=\Sol_G(x)/N$. In particular,
    \[
    \left|\frac{\Sol_G(x)}{N}\right| = \frac{|\Sol_G(x)|}{|N|}.
    \]
\end{lem}

In what follows, some famous results are needed that show the existence of a nilpotent maximal subgroup strongly affects the structure of a finite group. We collect them here for the reader convenience.

\begin{thm}[\cite{Janko}]\label{thm:janko}
Let $G$ be a finite group having a nilpotent maximal subgroup
$M$. If a Sylow $2$-subgroup of $M$ has class at most $2$, then $G$ is soluble.
\end{thm}

As a consequence we have the following.

\begin{cor}\label{cor:nilp}
Let $G$ be a finite group and let $M$ be a nilpotent subgroup of $G$. If the Sylow $2$-subgroup of $M$ has class at most $2$  and $M$ is not properly contained in any soluble subgroup of $G$, then $G=M$.
\end{cor}

In particular, Corollary~\ref{cor:nilp} holds when the subgroup $M$  has odd order.
Thus a central role in the description of finite groups with a nilpotent maximal subgroup is played by  Sylow $2$-subgroups. This is emphasized by the following.

\begin{thm}[\cite{Rose}]\label{thm:rose}
Suppose that $G$ is a finite insoluble group having a nilpotent maximal subgroup $M$. If $Z(G)=1$, then $M$ is a Sylow $2$-subgroup of $G$.
\end{thm}

Furthermore, in \cite{Baumann} Baumann showed that only some Sylow $2$-subgroups are admitted.

\begin{thm}[\cite{Baumann}]\label{thm:bau}
Let $G$ be a finite insoluble group having a nilpotent maximal subgroup. Let $L = \Fit(G)$ be the Fitting subgroup of $G$. Then $G/L$ has a unique minimal normal subgroup $K/L$, which is a direct product of copies of a simple group with dihedral Sylow $2$-subgroups, and $G/K$ is a $2$-group.
\end{thm}

On the other side, the groups with dihedral Sylow $2$-subgroups have been characterized in \cite{GW}. As a consequence we have the following.

\begin{thm}[\cite{GW}]\label{thm:gw}
If $G$ is a simple group with dihedral Sylow $2$-subgroups, then either $G$ is isomorphic to the projective special linear group $\PSL(2, q)$, $q$ odd and $q \geq 5$, or $G$ is isomorphic to the alternating group $\A(7)$. 
\end{thm}


\section{Proof of Theorem A}\label{A}

\noindent We start with the following crucial observation.

\begin{lem}\label{lem:nilpk}
Let $G$ be a group, $x \in G$ and $k \geq 2$. If $[u_1, \ldots, u_k]=1$ for every $u_1, \ldots, u_k \in \Sol_G(x)$, then $\Sol_G(x)$ is a subgroup. Moreover it is nilpotent of class at most $k-1$.
\end{lem}

\begin{proof}
To prove that $\Sol_G(x)$ is a subgroup, consider elements $y,z \in \Sol_G(x)$, and write $L=\langle x,y,z \rangle$. Then $\gamma_k(L)$ is generated by long commutators of weight at least $k$ with entry set $\{x,y,z,x^{-1},y^{-1},z^{-1}\}$ (see, for instance, \cite[2.1.5]{Khukhro}), and the latter are trivial by hypothesis. Thus $\gamma_k(L)=1$.
On the other hand, $\langle yz , x\rangle \leq L$ which is nilpotent of class at most $k-1$. Then $\Sol_G(x)$ is a subgroup and the result follows.
\end{proof}

Now we are in a position to prove Theorem~{A} stated in Section~\ref{I}.
\begin{thmA}
Let $G$ be a finite group. Then $G$ is nilpotent of class at most $2$ if and only if there exists an element $x \in G$ such that $[u_1,u_2,u_3]=1$ for every $u_1,u_2,u_3 \in \Sol_G(x)$.
\end{thmA}
\begin{proof}
By applying Lemma~\ref{lem:nilpk} for $k=3$, we get that $\Sol_G(x)$ is a nilpotent subgroup of class at most $2$.
Moreover, $R(G)=\Fit(G)$ as $R(G)$ is a subgroup of $\Sol_G(x)$.
Since $\Sol_G(x)$ is not properly contained in any soluble subgroup of $G$, Corollary~\ref{cor:nilp} yields that $G=\Sol_G(x)$, and we are done.
\end{proof}

It is not clear whether the answer to Question~\ref{q1} is in the affirmative for $k>3$. Nevertheless, if it is not the case, the structure of a minimal counterexample is subject to significant restrictions as showed in the following.

\begin{rem}
Let $k\geq 4$, and assume that $G$ is a minimal (with respect to the order) insoluble group having an element $x$ such that $[u_1, \ldots,u_k]=1$ for every $u_1,\ldots,u_k \in \Sol_G(x)$. Then $G$ has a unique minimal normal subgroup $K = S \times \cdots \times S$ where $S$ is  isomorphic either to $\PSL(2, q)$, $q$ odd and $q \geq 5$, or to the alternating group $\A(7)$. Furthermore, $G = K\langle x \rangle $ and $\Sol_G(x)$ is a Sylow $2$-subgroup of $G$.
\end{rem}

\begin{proof}
By Lemma~\ref{lem:nilpk}, $\Sol_G(x)$ is a subgroup of $G$ and it is nilpotent of class at most $k-1$. Hence $R(G)=\Fit(G)$. 
Set $T=\Sol_G(x)$. Since $G$ is not soluble, $T$ is a maximal subgroup by part $(b)$ of Lemma~\ref{lem:property}.
If $R(G) \neq 1$ then $G/R(G)$ is a group of smaller order with $\Sol_{G/R(G)}(xR(G))$ satisfying our hypotheses. This implies that $G/R(G)$ is  soluble, giving the contradiction that $G$ is soluble. Therefore $R(G)=\Fit(G)=Z(G)=1$. 
Now Theorem~\ref{thm:rose} implies that $T$ is a Sylow $2$-subgroup of $G$ of order at most $2^{k}$, and Theorem~\ref{thm:bau} yields the existence of a unique minimal normal subgroup $K$ of $G$ such that $K= S \times \cdots \times S$, where $S$ is a non-abelian simple group with dihedral Sylow $2$-subgroups, and $G/K$ is a $2$-group. Therefore, Theorem~\ref{thm:gw} implies that either $S$ is  isomorphic to $\PSL(2, q)$, $q$ odd and $q \geq 5$, or it is isomorphic to the alternating group $\A(7)$. 

By \cite[Theorem 2.13]{isaacs} $x$ is not an involution. Let $H= \langle x, K \rangle$. We claim that $G=H$. Indeed, if $H$ is a proper subgroup of $G$, then $\Sol_H(x)=H \cap \Sol_G(x)$ satisfies our hypotheses, so $H$ is soluble by the minimality of $G$. Therefore $G= K \langle x \rangle$.
Let $P= T \cap K$ be a Sylow $2$-subgroup of $K$. By the Dedekind's Modular Law we have
\[
T= G \cap T =  \langle x \rangle K  \cap T = \langle x \rangle (K  \cap T)= \langle x \rangle P.
\]
By Theorem~\ref{thm:bau}, $P$ is a direct product of dihedral groups.
\end{proof}


\section{Proof of Theorem B}\label{B}
\noindent In this section we deal with arithmetic questions related to the solubilizer of an element in a finite group.

\begin{lem}
Let $G$ be an insoluble group and let $x$ be an element of $G$ such that $\Sol_G(x)$ is a subgroup. Then $|\Sol_G(x)| \neq p^n$ for all odd primes $p$ and all positive integers $n$.
\end{lem}

\begin{proof}
Arguing by contradiction, assume that  $|\Sol_G(x)|= p^n$ for some odd prime $p$ and a positive integer $n$. Then $\Sol_G(x)$
is a Sylow $p$-subgroup of $G$. Moreover, $\Sol_G(x)$ is not properly contained in any soluble subgroup of $G$. From Corollary~\ref{cor:nilp} it follows that $G=\Sol_G(x)$, giving the contradiction that $G$ is soluble.
\end{proof}

The following result shows that the solubilizer of an element of prime order has to be large enough when it is not equal to the normalizer.

\begin{lem}\label{lem:bob}
Let $G$ be a finite group, and let $x \in G$ be an element of prime order $p$. If $P=\langle x \rangle $ and $|\Sol_G(x)|\leq p^2$, then $\Sol_G(x)=N_G(P)$.
\end{lem}

\begin{proof}
Since $N_G(P) \subseteq \Sol_G(x)$ by part $(a)$ of Lemma~\ref{lem:property}, we only need to prove the reverse inclusion.
 As $\Sol_G(x)$ is the union of all soluble subgroups of $G$ containing $x$, it suffices to show that
if $H$ is a soluble subgroup containing $x$, then $H \leq N_G(P)$. Let $H$ be a soluble subgroup containing $x$. Then $|H| \leq p^2$. If $H$ is a $p$-group, then $H$ is abelian, and so $H \leq N_G(P)$. Therefore assume that $H$ is not a $p$-group. It follows that $|H : P| < p$, and $|H:N_H(P)|$ is congruent to $1$ modulo $p$ because  $P$ is a Sylow $p$-subgroup of $H$. Hence, $H = N_H(P) \leq N_G(P)$. This yields $\Sol_G(x)  \leq  N_G(P)$ as desired.
\end{proof}

Now we are in a position to prove Theorem~{B} stated in Section~\ref{I}.
\begin{thmB}
Let $G$ be an insoluble group and $x$ an element of $G$. Then the cardinality of $\Sol_G(x)$ cannot be equal to $p^2$ for any prime $p$.
\end{thmB}
\begin{proof}
Arguing by contradiction, assume that $|\Sol_G(x)|=p^2$ for some prime $p$. It follows that $o(x)$ divides $p^2$, and $x \in P$ for some Sylow $p$-subgroup $P$ of $G$. Clearly, $p^2$ does not divide $|G|$, otherwise $P = \Sol_G(x)$ and $G$ is abelian by \cite[Theorem~1.2]{ALMM}. 
Therefore, we can assume that $P= \langle x \rangle $ has cardinality $p$. By Lemma~\ref{lem:bob} we get $\Sol_G(x)=N_G(P)$, which is a contradiction because there are no subgroups of cardinality $p^2$ in $G$. The proof is complete.
\end{proof}

As a consequence of Theorem~B we have the following.

\begin{cor}\label{cor:cubo}
Let $G$ be an insoluble group and $x$ an element of $G$. If $R(G) \neq 1$, then $|\Sol_G(x)| \neq p^3$ for all primes $p$.
\end{cor}

\begin{proof}
Arguing by contradiction, assume that $|\Sol_G(x)| = p^3$ for some $x \in G$ and $p$ prime. Then, from Lemma~\ref{quotient} it follows that $\Sol_{G/R(G)}(xR(G))$ has order $p$ or $p^2$, which is a contradiction by Theorem~B and \cite[Corollary 3.2 $(b)$ and Corollary 3.3]{ALMM}.
\end{proof}

Actually we suspect that Corollary~\ref{cor:cubo} holds true also when $R(G)=1$. As a consequence of Theorem A we are able to prove the above conjecture for $p=2$.

\begin{pro}\label{pro:eight}
Let $G$ be an insoluble group. Then $|\Sol_G(x)| \neq 8$ for every element $x \in G$.
\end{pro}

\begin{proof} Assume $|\Sol_G(x)|=8$. Since $G$ is not soluble, $R(G) \neq G$. Moreover, if $R(G) \neq 1$ then $\Sol_{G/R(G)}(xR(G))= \Sol_G(x)/R(G)$ has cardinality $2$ or $4$, which is impossible by Theorem~B and \cite[Theorem~1.2]{ALMM}. Hence we can assume $R(G)=1$. As $o(x)$ divides $|\Sol_G(x)|$, we can consider a Sylow $2$-subgroup of $G$, say $P$, containing $x$. Since $G$ is insoluble, it follows that $G$ is not $2$-nilpotent. Thus, applying \cite[10.1.9]{Robinson}, we can assume that $P$ is not cyclic and $4 \leq |P| \leq 8$.

By Frobenius' normal $p$-complement theorem  (see, for instance, \cite[Theorem~7.4.5]{GOR}), there exist a $2$-subgroup $H$ of $G$ and an element $b \in N_G(H) \setminus C_G(H)$ of odd order. Then $|H| \geq 4$, otherwise $N_G(H)=C_G(H)$.

As $H \leq P^g \leq \Sol_G(x)^g=\Sol_G(x^g)$ for some $g\in G$, without loss of generality we can assume that $H \leq P$.
If $H=P$, then $\langle H,b \rangle \subseteq \Sol_G(x)$, which is impossible since $|\langle H,b \rangle| >8 $. Then $|H|=4$ and $|P|=8$, which implies $\Sol_G(x)=P$. 
 Then $\Sol_G(x)$ has nilpotency class at most $2$, so $G$ is nilpotent of class at most $2$ by Theorem~A, which is a contradiction. This concludes the proof.
\end{proof}

\begin{lem}\label{lem:six}
Let $G$ be a finite group with trivial soluble radical and let $x$ be a self-centralizing element of $G$ of order $3$. Then $|\Sol_G(x)| \in \{24,78\}$.
\end{lem}

\begin{proof}
Since $R(G)=1$, by the main theorem of \cite{FT} we deduce that $G$ is isomorphic either to the alternating group $\A(5)$, or to the projective special linear group $\PSL(2,7)$. Since all elements of order $3$ are conjugate in the above groups, it is sufficient to observe that in the former case $|\Sol_{G}(x)|=24$ for $x=(1,2,3)$, while in the latter $|\Sol_G(x)|=78$ for $$
x=\begin{pmatrix}
0 & 1 & 0\\
1 & 1 & 0\\
0 & 0 & 1
\end{pmatrix}.
$$
\end{proof}

\begin{pro}\label{pro:six}
Let $G$ be an insoluble group. Then $|\Sol_G(x)| \neq 6$ for every element $x \in G$.
\end{pro}

\begin{proof}
Assume that $|\Sol_G(x)|=6$. Then $o(x)$ divides $6$. By \cite[Theorem~1.2]{ALMM}, we can assume $o(x) \neq 6$. If $o(x)=3$, Lemma~\ref{lem:bob} implies that $\Sol_G(x)=N_G(x)$. In particular $C_G(x)=\langle x \rangle$ and $x$ is a self-centralizing element of $G$. Moreover, the soluble radical $R(G)$ of $G$ is trivial, otherwise $\Sol_{G/R(G)}(xR(G))$ would have prime order. Then, by Lemma~\ref{lem:six} we can conclude that $|\Sol_G(x)| \neq 6$, which is a contradiction.

Now assume $o(x)=2$. Let $P$ be a Sylow $2$-subgroup of $G$ containing $x$. Since $G$ is not $2$-nilpotent and $|\Sol_G(x)|=6$, $P$ is elementary abelian of order $4$. Now, for every $y \in G$ such that $o(y)=2$, the group $\langle x,y \rangle$ is a dihedral group and thus $y\in\Sol_G(x)$. Therefore every Sylow $2$-subgroup of $G$ is contained in $\Sol_G(x)$. Let $n_2$ be the number of Sylow $2$-subgroups of $G$. Of course $n_2 >1$ because $G$ is not soluble. Hence $n_2 \geq 3$, and $\Sol_G(x)$ contains more than $6$ elements, our final contradiction.
\end{proof}

As a consequence of Theorem~B, Propositions~\ref{pro:eight} and \ref{pro:six} we have the following.

\begin{cor}
Let $G$ be an insoluble group. Then $|\Sol_G(x)| \geq 10$ for all elements $x \in G$.
\end{cor}

We point out that if $G=\A (5)$ is the alternating group any element $x \in G$ of order $5$  has $\Sol_G(x)=N_G(x)$ with $|N_G(x)|=10$. Therefore it is possible to have $|\Sol_G(x)|=pq$ with $p>q$ primes.
However there are some restrictions, as showed in the following.

\begin{pro}
Let $G$ be an insoluble group, and let $x$ be any element of $G$ such that $|\Sol_G(x)|=3p$ where $p$ is a prime. Then $o(x)=p$, $p \equiv 1 \pmod 3$ and $\Sol_G(x)=N_G(x)$.
\end{pro}

\begin{proof}
First of all, we can assume $R(G)=1$, otherwise $\Sol_{G/R(G)}(xR(G))$ is a prime against \cite[Corollary 3.2 $(b)$]{ALMM}. By Proposition~\ref{pro:six} and Theorem~\ref{thm:quadrato} we can assume $p \geq 5$. Assume $o(x)=3$. Since $|C_G(x)|$ divides $3p$, we have $|C_G(x)|=3$ and we get a contradiction by Lemma~\ref{lem:six}. Therefore $o(x)=p$ and from Lemma~\ref{lem:bob} it follows that $\Sol_G(x)=N_G(x)$. Finally, by \cite[Theorem~1.2]{ALMM} we obtain $p \equiv 1 \pmod 3$, and we are done.  
\end{proof}

\section*{Acknowledgements}
\noindent The first author was partially supported by a grant from the Niels Hendrik Abel Board, and she would like to thank the International Mathematical Union. The second and the third authors are members of the National Group for Algebraic and Geometric Structures, and their Applications  (GNSAGA -- INdAM). 
This work was carried out during the first author's visit to the University of Salerno. She wishes to thank the Department of Mathematics for the excellent hospitality.
Finally, the authors are very grateful to the referee for the insightful comments, valuable for the improvement of this work.

\section*{Data Availability Statement}
\noindent This manuscript has no associate data.

\end{document}